\documentclass{article}
\usepackage{amssymb}
\usepackage{verbatim}
\usepackage{array}
\usepackage{latexsym}
\usepackage{enumerate}
\usepackage{amsmath}
\usepackage{amsfonts}
\usepackage{amsthm}
\usepackage{color}
\usepackage[english]{babel}
\usepackage{tikz}
 \usetikzlibrary{automata,topaths}
\usepackage{float}

\newtheorem{theorem}{Theorem}[section]
\newtheorem{proposition}[theorem]{Proposition}
\newtheorem{lemma}[theorem]{Lemma}
\newtheorem{corollary}[theorem]{Corollary}

\newtheorem{observation}[theorem]{Observation}

\newtheorem{example}[theorem]{Example}

\newtheorem{conj}[theorem]{Conjecture}
\def\x{\mathbf x}
\def\y{\mathbf y}

\def\cC{\mathcal C}

\def\cH{\mathcal H}

\def\cX{\mathcal X}

\newcommand{\fqq}{\mathbb {F}_{q^2}}
\newcommand{\fqs}{\mathbb {F}_{q^6}}

\def\K{\mathbb{K}}

\title{Weierstrass semigroups on the Giulietti--Korchm\'aros curve}
\date{}
\author{Peter Beelen and Maria Montanucci}
\begin{document}
\maketitle

\begin{abstract}
In this article we explicitly determine the structure of the Weierstrass semigroups $H(P)$ for any point $P$ of the Giulietti--Korchm\'aros curve $\cX$. We show that as the point varies, exactly three possibilities arise: One for the $\mathbb{F}_{q^2}$-rational points (already known in the literature), one for the $\mathbb{F}_{q^6} \setminus \mathbb{F}_{q^2}$-rational points, and one for all remaining points. As a result, we prove a conjecture concerning the structure of $H(P)$ in case $P$ is a $\mathbb{F}_{q^6} \setminus \mathbb{F}_{q^2}$-rational point. As a corollary we also obtain that the set of Weierstrass points of $\cX$ is exactly its set of $\mathbb{F}_{q^6}$-rational points.
\end{abstract}

\thanks{{\em Math.~Subj.~Class.:} Primary:  11G20.     \    Secondary:  11R58, 14H05, 14H55.  }


\section{Introduction}

Let $\cC$ be a nonsingular, projective algebraic curve of genus $g$ defined over a field $\mathbb{F}$. Let $P$ be a rational point on $\cC$. The \textit{Weierstrass semigroup} $H(P)$ is defined as the set of integers $k$ such that there exists a function on $\cC$ having pole divisor exactly $kP$. More generally $H(P)$ can be defined for any point $P$ on $\cC$ by considering $\cC$ as an algebraic curve over the algebraic closure of $\mathbb{F}$. It is clear that $H(P)$ is a subset of natural numbers $\mathbb{N}=\{0,1,2,\ldots\}$. The Weierstrass gap Theorem, see \cite[Theorem 1.6.8]{Sti}, states that the set $G(P):= \mathbb{N} \setminus H(P)$ contains exactly $g$ elements, which are called \textit{gaps}. The structure of $H(P)$ is not always the same for every point $P$ of $\cC$. However, it is known that for generically the semigroup $H(P)$ is the same, but there can exist finitely many points of $\cC$, called \textit{Weierstrass points}, with a different gap set. These points are of intrinsic interest, for example in St\"ohr--Voloch theory \cite{SV}, but in case $\mathbb{F}=\mathbb{F}_q$, the finite field with $q$ elements, they also occur in the study of algebraic geometry (AG) codes \cite{TV1991}. In this context, a commonly studied class of curves are the so-called \textit{maximal curves}, that is, algebraic curves defined over a finite field $\mathbb{F}_q$ having as many rational points as possible according to the Hasse--Weil bound. More precisely, an algebraic curve $\cC$ with genus $g(\cC)$ and defined over $\mathbb{F}_q$ is said to be an $\mathbb{F}_{q}$-maximal curve if it has $q+1+2g(\cC)\sqrt{q}$ points defined over $\mathbb{F}_q$. Clearly, this can only be the case if the cardinality $q$ of the finite field is a square.

An important and well-studied example of an $\mathbb{F}_{q^2}$-maximal curve is given by the Hermitian curve $\mathcal{H}$.
%It is a plane curve, which can be defined by the affine equation
%$$
%Y^{q+1}=X^q+X.
%$$
For fixed $q$, the curve $\mathcal{H}$ has the largest possible genus $g(\mathcal{H}) =q(q-1)/2$ that an $\mathbb{F}_{q^2}$-maximal curve can have. The Weierstrass points on $\mathcal H$ and the precise structure of the semigroups for $P$ on $\mathcal{H}$ are known; see \cite{GV}. By a result commonly attributed to Serre, see \cite[Proposition 6]{L1987}, any $\mathbb{F}_{q^2}$-rational curve which is covered by an $\mathbb{F}_{q^2}$-maximal curve is also $\mathbb{F}_{q^2}$-maximal. Most of the known maximal curves are subcovers of the Hermitian curve. The first known example of a maximal curve which is not a subcover of the Hermitian curve was constructed by Giulietti and Korchm\'aros; see \cite{GK}. This curve is an $\mathbb{F}_{q^6}$-maximal curve and commonly called the Giulietti--Korchm\'aros (GK) curve. The aim of this paper is to complete the description of the Weierstrass semigroups occurring for this curve.

The Weierstrass semigroup for any $\mathbb{F}_{q^2}$-rational point of $\cX$ was computed in \cite{GK}, but the structure of the Weierstrass semigroup $H(P)$ where $P \not\in \cX(\mathbb{F}_{q^2})$ is not known, except for $q \le 9$, \cite{FG2010,D2011}. Based on the available data for small $q$, a conjecture concerning the structure of $H(P)$ was stated in \cite{D2011} for $P \in \cX(\mathbb{F}_{q^6}) \setminus \cX(\mathbb{F}_{q^2})$. For $P \not\in \cX(\fqs)$ nothing specific is known about $H(P).$ In this article we determine settle the conjecture from \cite{D2011} and also determine the structure of the generic semigroup for $P$ on $\cX$. More precisely, we show the following theorem.
\begin{theorem} \label{mainth}
Let $q$ be a prime power and let $P$ be a point of the Giulietti--Korchm\'aros curve $\cX$.
The Weierstrass semigroup $H(P)$ is given by
\begin{itemize}
\item $H(P)=\langle q^3 -q^2 +q, q^3, q^3 + 1 \rangle,$ \ if $P \in \cX(\mathbb{F}_{q^2})$;
\item $H(P)=\langle q^3-q+1,q^3+1,q^3+i(q^4-q^3-q^2+q-1) \mid i=0,\ldots,q-1\rangle,$
\ if $P \in \cX(\mathbb{F}_{q^6}) \setminus \cX(\mathbb{F}_{q^2})$;
\item $H(P)=\mathbb{N} \setminus G,$ \ if $P \not\in \cX(\mathbb{F}_{q^6})$, where
$$G=\left\{iq^3+kq+m(q^2+1)+\sum_{s=1}^{q-2} n_s ((s+1)q^2)+j+1 \mid i,j,k,m,n_1,\ldots,n_{q-2} \in \mathbb{Z}_{\geq0}, \ j \le q-1 \ \makebox{and}\right.$$
$$\left.i+j+k+mq+\sum_{s=1}^{q-2} n_s ((s+1)q-s) \leq q^2-2\right\}.$$
\end{itemize}
\end{theorem}
As mentioned above, the case $P \in \cX(\mathbb{F}_{q^2})$ is already known and taken from \cite{GK}. As a bonus, we will also obtain the set of Weierstrass points of $\cX$.
\begin{corollary} \label{mainth2}
Let $W$ denote the set of Weierstrass points of the Giulietti--Korchm\'aros curve $\cX$. Then $W=\cX(\mathbb{F}_{q^6})$.
\end{corollary}

The paper is organized as follows: In the next section we give the necessary background on the GK curve as well as some results on Weierstrass semigroups and their gaps that we will need later. In section three, we settle the conjecture from \cite{D2011} concerning $H(P)$ for $P \in \cX(\mathbb{F}_{q^6}) \setminus \cX(\mathbb{F}_{q^2})$, while in section four, we compute the Weierstrass semigroup for $P \not \in \cX(\mathbb{F}_{q^6})$. We finish with some concluding remarks and observations.

\section{The Giulietti--Korchm\'aros curve} \label{sec2}

Let $q$ be a prime power and $\K=\overline{\mathbb{F}}_q$. The Giulietti--Korchm\'aros (GK) curve $\cX$ is a non-singular curve in ${\rm PG}(3,\K)$ defined by the affine equations
\begin{equation}\label{eq:GK}
\cX: \left\{
\begin{array}{l}
Y^{q+1}=X^q+X,\\
Z^{q^2-q+1}=Y^{q^2}-Y.\\
\end{array}
\right.
\end{equation}
This curve has genus $g(\cX)=(q^5-2q^3+q^2)/2$ and $q^8-q^6+q^5+1$ $\mathbb{F}_{q^6}$-rational points. The curve $\cX$ has been introduced in \cite{GK}, where it was proved that $\cX$ is maximal over $\fqs$, that is, the number $|\cX(\fqs)|$ of $\fqs$-rational points of $\cX$ equals $q^6+1+2gq^3$.
Also, for $q>2$, the curve $\cX$ is not $\fqs$-covered by the Hermitian curve maximal over $\fqs$; $\cX$ was the first maximal curve shown to have this property. Note that equation \eqref{eq:GK} implies that $\cX$ is a cover of the Hermitian curve over $\fqq$ given by the affine equation $Y^{q+1}=X^q+X$. We will denote this curve by $\cH$.
%Also, for $q>2$, it is the unique example of maximal curve which is not $\mathbb{F}_{q^2}$-covered by the Hermitian curve $\mathcal{H}(\mathbb{F}_{q^2})$.

The automorphism group ${\rm Aut}(\cX)$ of $\cX$ is defined over $\fqs$ and has order $q^3(q^3+1)(q^2-1)(q^2-q+1)$.  Moreover, it has a normal subgroup isomorphic to ${\rm SU(3,q)}$, the automorphism group of the Hermitian curve $\cH$. %If $(3,q+1)=1$, then ${\rm Aut}(\cX)\cong{\rm SU}(3,q)\times C_{q^2-q+1}$, where $C_{q^2-q+1}$ is a cyclic group of order $q^2-q+1$. If $(3,q+1)=3$, then ${\rm SU}(3,q)\times C_{(q^2-q+1)/3}$ is isomorphic to a normal subgroup of ${\rm Aut}(\cX)$ of index $3$; see \cite[Theorem 6]{GK}.
The set $\cX(\fqs)$ of the $\mathbb{F}_{q^6}$-rational points of $\cX$ splits into two orbits under the action of ${\rm Aut}(\cX)$:
one orbit $\mathcal O_1=\cX(\fqq)$ of size $q^3+1$, which coincides with the intersection between $\cX$ and the plane $Z=0$; and another orbit $\mathcal O_2=\cX(\fqs)\setminus\cX(\fqq)$ of size $q^3(q^3+1)(q^2-1)$; see \cite[Theorem 7]{GK}. The orbits $\mathcal O_1$ and $\mathcal O_2$ are the short orbits of ${\rm Aut}(\cX)$, that is, the unique orbits of points of $\cX$ having a non-trivial stabilizer in ${\rm Aut}(\cX)$.

Let $x,y,z\in \K(\cX)$ be the coordinate functions of the function field of $\cX$, which satisfy $y^{q+1}=x^q+x$ and $z^{q^2-q+1}=y^{q^2}-y$. Then we denote by $P_{(a,b,c)}$ the affine point of $\cX$ with coordinates $(a,b,c)$ and by $P_\infty$ the unique point at infinity. Similarly, we denote by $Q_{(a,b)}$ the affine point of the Hermitian curve $\cH$ with coordinates $(a,b)$ and by $Q_\infty$ its unique point at infinity.

The Weierstrass semigroup at $P_\infty$, and hence at every $\mathbb{F}_{q^2}$-rational point of $\cX$ (since they lie in the same short orbit $\mathcal O_1$ of ${\rm Aut}(\cX)$) was computed in \cite{GK}.
\begin{proposition} \rm{\cite[Proposition 6.2]{GK}}
The Weierstrass semigroup of $\cX$ at $P_\infty$ is generated by $q^3 -q^2 +q$, $q^3$, $q^3 + 1$.
\end{proposition}

Before describing what is known about $H(P)$ for $P \not \in \cX(\fqq)$, we introduce several functions on $\cX$ and give their divisors. Some of these functions can be interpreted as functions on $\cH$ as well and therefore have a divisor on $\cH$. To differentiate, we will write $(f)_\cH$ (resp. $(f)_\cX$) for divisors on the Hermitian curve $\cH$ (resp. the GK curve $\cX$).
Given a point $P=P_{(a,b,c)}$ on $\cX$, we define the functions
\begin{equation} \label{tilda}
\tilde{x}_{P}=-a^q-x+b^qy, \quad \tilde{y}_P=y-b, \quad \tilde{z}_P=-a^{q^3}-x+b^{q^3}y+c^{q^3}z.
\end{equation}
%Note that the functions $\tilde{x}_{P}$ and $\tilde{y}_{P}$ also can be interpreted as functions on $\cH$. 
Then it is not hard to show the following. %Compare \cite[Thm.9.79]{HKT}.
\begin{align}
% (\tilde{x}_{P})_\cH & =qQ_{(a,b)}+Q_{(a^{q^2},b^{q^2})}-(q+1)Q_\infty,\label{eq:divxpH}\\
 (\tilde{x}_{P})_\cX &=q\sum_{\xi^{q^2-q+1}=1}P_{(a,b,\xi c)}+\sum_{\xi^{q^2-q+1}=1}P_{(a^{q^2},b^{q^2},\xi c^{q^2})}-(q^3+1)P_{\infty},\label{eq:divxpX}\\
% (\tilde{y}_{P})_\cH& =\sum_{s^q+s=0} Q_{(a+s,b)}-qQ_\infty,\label{eq:divypH}\\
 (\tilde{y}_{P})_\cX & =\sum_{s^q+s=0, \  \xi^{q^2-q+1}=1} P_{(a+s,b,\xi c)}-(q^3-q^2+q)P_\infty,\label{eq:divypX}\\
 (\tilde{z}_{P})_\cX & =q^3 P_{(a,b,c)}+{P_{(a^{q^6},b^{q^6},c^{q^6})}}-(q^3+1)P_\infty,\label{eq:divzpX}\\
 (z)_\cX & =\sum_{P\in\cX(\fqq),P\ne P_\infty} P \, - \, q^3P_{\infty}.\label{eq:divz}
\end{align}

Now let $P=P_{(a,b,c)}$ be a fixed $\mathbb{F}_{q^6}$-rational point of $\cX$ which is not $\mathbb{F}_{q^2}$-rational (implying $c \neq 0$). In this case equation \eqref{eq:divzpX} implies:
\begin{equation}\label{eq:divzpX6}
(\tilde{z}_{P})_\cX=(q^3+1)(P-P_\infty) \ \makebox{for} \ P=P_{(a,b,c)} \in \cX(\fqs).
\end{equation}
The Weierstrass semigroup $H(P)$ is only completely known in finitely many cases if $P\in \cX(\fqs)\setminus \cX(\fqq)$. It was computed for $q=2$ and $q=3$ in \cite{FG2010} and for $4 \le q \le 9$ in \cite{D2011}. Also in \cite{D2011}, the following partial information was obtained for general $q$: Equations \eqref{eq:divxpX}, \eqref{eq:divypX} and \eqref{eq:divzpX6} imply that the functions $1/\tilde z_P, \tilde y_P / \tilde z_P, \tilde x_P / \tilde z_P$ have poles only in $P$ of orders $q^3+1$, $q^3$ and $q^3-q+1$ respectively. Hence
\begin{equation}\label{eq:somepoles}
\langle q^3-q+1,q^3,q^3+1\rangle \subseteq H(P) \ \makebox{for} \ P \in \cX(\fqs) \setminus \cX(\fqq).
\end{equation}
Based on this and the results for $q \le 9$, the following conjecture was stated in \cite{D2011}, which we will prove in the next section.
\begin{conj}\label{conjD}
The Weierstrass semigroup $H(P)$ of $\cX$ at $P \in \cX(\mathbb{F}_{q^6}) \setminus \cX(\mathbb{F}_{q^2})$ is given by
$$H(P)=\langle q^3-q+1,q^3+1,q^3+i(q^4-q^3-q^2+q-1) \mid i=0,\ldots,q-1\rangle.$$
\end{conj}

Finally, for $P \not \in \cX(\fqs)$ nothing specific is known about the structure of semigroup $H(P)$. We will completely determine its gap structure, but for now, we finish this section by stating some facts that we will use to achieve this. We start with the following well-known lemma connecting regular differentials (i.e., differential forms having no poles anywhere on $\cX$) and gaps of $H(P)$.

\begin{proposition}{\rm \cite[Corollary 14.2.5]{VS}}\label{prop:holom}
Let $\cX$ be an algebraic curve of genus $g$ defined over $\K$. Let $P$ be a point of $\cX$ and $\omega$ be a regular differential on $\cX$. Then $v_P(\omega)+1$ is a gap at $P$.
\end{proposition}
This proposition has the following, for us very useful, consequence.
\begin{corollary}\label{holom}
For any point $P$ on the GK curve $\cX$ distinct from $P_\infty$ and for any $f \in L((2g(\cX)-2)P_\infty),$ we have $v_P(f)+1 \in \mathbb{N} \backslash H(P).$
\end{corollary}
\begin{proof}
First note that $(dy)_\cH=(q^2-q-2)Q_\infty$. The set of points that ramify in the covering of $\cX$ by $\cH$ is exactly $\cH(\fqq)$, the set of $\fqq$-rational points of the Hermitian curve, all with ramification index $q^2-q+1$. Moreover, the points of $\cX$ above $\cH(\fqq)$ are precisely the $\fqq$-rational points of $\cX$.  Therefore, we immediately obtain that
$$(dy)_\cX=(q^4-2q^3+q^2-2)P_\infty+(q^2-q)\sum_{P\in\cX(\fqq),P\ne P_\infty} P.$$
Thus, from $z^{q^2-q+1}=y^{q^2}-y$ and equation \eqref{eq:divz},
$$(dz)_\cX=(-dy/z^{q^2-q})_\cX=(q^5-2q^3+q^2-2)P_\infty.$$
In particular a differential $fdz$ is regular if and only if $f \in L((q^5-2q^3+q^2-2)P_\infty)=L((2g(\cX)-2)P_\infty)$.
The corollary now follows by applying Proposition \ref{prop:holom}.
\end{proof}

\section{The Weierstrass semigroup $H(P)$ for $P \in \cX(\mathbb{F}_{q^6}) \setminus \cX(\mathbb{F}_{q^2})$}

This section is devoted to the proof of Conjecture \ref{conjD} for any prime power $q$. In particular in this section $P=P_{(a,b,c)}$ will always denote a point in $\cX(\mathbb{F}_{q^6}) \setminus \cX(\mathbb{F}_{q^2})$. Further we define the semigroup $$T:=\langle q^3-q+1,q^3+1, q^3+i(q^4-q^3-q^2+q-1) \mid i=0,\ldots,q-1\rangle.$$ Conjecture \ref{conjD} then simply states that $H(P)=T$.
Our proof of the conjecture consists of two main steps. In the first step, we will show that $T \subset H(P)$ by showing that the generators of $T$ are in $H(P)$. In the second step, we show that the number of gaps of the semigroup $T$ (also known as the genus of $T$) is exactly equal to the genus of $\cX$. Once this has been established, the equality $H(P)=T$ will follow immediately, proving Conjecture $\ref{conjD}$.

\subsection{$T \subset H(P)$}

As before we use the function $\tilde x_P$ defined in equation \eqref{tilda} and its divisor in equation \eqref{eq:divxpX}. Moreover, for $k \in \mathbb{Z}$, we define the $k$-th Frobenius twist of $\tilde x_P$ as the follows:
\begin{equation}\label{eq:FrobxP}
\tilde x_P^{(k)}:=-a^{q^{2k+1}}-x+b^{q^{2k+1}}y \ \makebox{for} \ P=P_{(a,b,c)}.
\end{equation}
Since we assume that $P \in \cX(\mathbb{F}_{q^6}) \setminus \cX(\mathbb{F}_{q^2}),$ equation \eqref{eq:divxpX} implies that
\begin{align}\label{eq:divFrobxP}
(\tilde x_P^{(1)})_{\cX} & =q\sum_{\xi^{q^2-q+1}=1}P_{(a^{q^2},b^{q^2},\xi c^{q^2})}+\sum_{\xi^{q^2-q+1}=1}P_{(a^{q^4},b^{q^4},\xi c^{q^4})}-(q^3+1)P_{\infty},\notag\\
(\tilde x_P^{(2)})_{\cX} & =q\sum_{\xi^{q^2-q+1}=1}P_{(a^{q^4},b^{q^4},\xi c^{q^4})}+\sum_{\xi^{q^2-q+1}=1}P_{(a,b,\xi c)}-(q^3+1)P_{\infty}.
\end{align}

\begin{lemma} \label{functions}
Let $P=P_{(a,b,c)} \in \cX(\mathbb{F}_{q^6}) \setminus \cX(\mathbb{F}_{q^2})$ and let $\tilde f_i=f_i / \tilde z_P^{iq-i+1}$ where
\begin{equation*} %\label{fi}
f_i:=\frac{(\tilde x_P)^{qi} \cdot \tilde x_P^{(2)}}{(\tilde x_P^{(1)})^{i}}, \ \makebox{for} \ i=1,\dots,q-1.
\end{equation*}
Then $(\tilde f_i)_{\infty}=(q^3+i(q^4-q^3-q^2+q-1))P$ and in particular $q^3+i(q^4-q^3-q^2+q-1) \in H(P)$ for $i=1,\ldots,q-1$.
\end{lemma}
\begin{proof}
Using equations \eqref{eq:divxpX} and \eqref{eq:divFrobxP}, we directly obtain that
$$(f_i)_{\cX}=(iq^2+1)\sum_{\xi^{q^2-q+1}=1}P_{(a,b,\xi c)}+(q-i)\sum_{\xi^{q^2-q+1}=1}P_{(a^{q^4},b^{q^4},\xi c^{q^4})}-(q^3+1)(iq-i+1)P_\infty.$$
Now using the divisor of $\tilde z_P$ given in equation \eqref{eq:divzpX6}, we find that
$$(\tilde f_i)_{\cX}=-(q^3+i(q^4-q^3-q^2+q-1))P+(iq^2+1)\sum_{\substack{\xi^{q^2-q+1}=1, \\ \xi \neq 1}}P_{(a,b,\xi c)}+(q-i)\sum_{\xi^{q^2-q+1}=1}P_{(a^{q^4},b^{q^4},\xi c^{q^4})}.$$
The lemma now follows.
\end{proof}

Note that the lemma is also true for $i=0$. Considering the corresponding function $\tilde f_0=\tilde x_P^{(2)}/\tilde z_P$, gives a way to show that $q^3 \in H(P)$. However, this is already known, see equation \eqref{eq:somepoles}.
\begin{proposition}\label{prop:contained}
Let $P \in \cX(\mathbb{F}_{q^6}) \setminus \cX(\mathbb{F}_{q^2})$. Then $T\subset H(P).$
\end{proposition}
\begin{proof}
Equation \eqref{eq:somepoles} and Lemma \ref{functions} imply that $\{ q^3-q+1,q^3+1, q^3+i(q^4-q^3-q^2+q-1) \mid i=0,\ldots,q-1\} \subset H(P)$. Since by definition these numbers generate $T$, the proposition follows.
\end{proof}

\subsection{The genus of the numerical semigroup $T$ equals $g(\cX)$}

We now show that the genus $g(T)$ of the numerical semigroup $T=\langle q^3-q+1,q^3+1,q^3+i(q^4-q^3-q^2+q-1) \mid i=0,\ldots,q-1\rangle$ is equal to $g(\cX)=(q^5-2q^3+q^2)/2$. In this way, since we already know that $T \subseteq H(P_{(a,b,c)})$ from Proposition \ref{prop:contained}, Conjecture \ref{conjD} will be completely proved.
We recall that a numerical semigroup is called \textit{telescopic} if it is generated by a telescopic sequence, that is by a sequence $(a_1,\ldots,a_k)$ such that
\begin{itemize}
\item $\gcd(a_1, \ldots , a_k)=1$;
\item for each $i=2,\ldots,k$, $a_i/d_i \in \langle a_1/d_{i-1},\ldots, a_{i-1}/d_{i-1}\rangle$, where $d_i=\gcd(a_1,\ldots,a_i)$ and $d_0=0$;
\end{itemize}
see \cite{KP1995}. From \cite[Proposition 5.35]{HVP1998}, the genus of a semigroup $\Gamma$ generated by a telescopic sequence $(a_1,\ldots,a_k)$ is
\begin{equation} \label{gentelescopic}
g(\Gamma)=\frac{1}{2} \bigg( 1+ \sum_{i=1}^k \bigg( \frac{d_{i-1}}{d_i}-1\bigg) a_i \bigg).
\end{equation}
For the semigroup $S$ defined by $S:=\langle q^3-q+1,q^3+1\rangle$ we obtain the following:
\begin{lemma} \label{genusS}
The numerical semigroup $S=\langle q^3-q+1,q^3+1\rangle$ is telescopic. Its genus $g(S)$ is given by
$$g(S)=\frac{q^3(q^3-q)}{2}.$$
\end{lemma}
\begin{proof}
Let $a_1=q^3-q+1$ and $a_2=q^3+1$. Then $gcd(a_1,a_2)=1$ and, using the same notation as above, $d_1=a_1$ and $d_2=1$. Since $a_2/d_2 \in \langle 1 \rangle= \langle a_1/d_1 \rangle$, $S$ is telescopic.
Thus from equation \eqref{gentelescopic},
$$g(S)=\frac{1}{2} \bigg( 1-a_1+(a_1-1)a_2 \bigg)=\frac{q^3(q^3-q)}{2}.$$
\end{proof}

Now the idea is to compute the number of gaps of $T$ by identifying the elements of $T$ that are gaps of $S$. The following observation is trivial, but will be very useful.
\begin{observation}\label{obs:representationab}
For any integer $n$, there exist unique integers $a$ and $b$ such that $n=a(q^3-q+1)+b(q^3+1)$ and $0 \le b \le q^3-q.$ An integer $n$ is an element of the semigroup $S=\langle q^3-q+1,q^3+1\rangle$ if and only if there exist integers $a$ and $b$ such that $n=a(q^3-q+1)+b(q^3+1)$, $a \ge 0$ and $0 \le b \le q^3-q.$
\end{observation}
In the following lemma, we identify several elements of $T \setminus S$ that turn out to play an important role.
\begin{lemma}\label{lem:sij}
For any $i=0,\ldots,q-1$ and $j=1,\ldots,q-1$, define the set
$$S_{i,j}:=\{(iq-jq^2+k_1)(q^3-q+1)+(jq^2-i+k_2)(q^3+1) \mid k_1=0,\ldots,q-1, \ k_2=0,\ldots,q^3-q-jq^2+i\}.$$
Then we have:
\begin{enumerate}
\item $S_{i,j} \subset T \setminus S.$
\item $S_{i,j} \cap S_{i'j'} = \emptyset$ if $(i',j') \neq (i,j)$, $0 \le i' \le q-1$ and $1 \le j' \le q-1.$
\item $|S_{i,j}|=q(q^3-q-jq^2+i+1).$
\end{enumerate}
\end{lemma}
\begin{proof}
First of all note that
\begin{equation*}%\label{eq:reprij}
jq^3+i(q^4-q^3-q^2+q-1)=(-jq^2+iq)(q^3-q+1)+(jq^2-i)(q^3+1).
\end{equation*}
Using this, it is clear from Proposition \ref{prop:contained}, that $(iq-jq^2+k_1)(q^3-q+1)+(jq^2-i+k_2)(q^3+1) \in T$ for any $i,j,k_1,k_2$ in the given range. To show that these elements are not in $S$, observe that
\begin{equation}\label{eq:ijk}
iq-jq^2+k_1 \le (q-1)q-q^2+q-1<0 \ \makebox{and} \ 0 \le jq^2-i+k_2 \le q^3-q.
\end{equation}
Observation \ref{obs:representationab} now implies that $(iq-jq^2+k_1)(q^3-q+1)+(jq^2-i+k_2)(q^3+1) \not \in S.$
This completes the proof of the first item.

Now suppose that $S_{i,j} \cap S_{i'j'} \neq \emptyset$. Then there exist integers $k_1,k_1',k_2,k_2'$ satisfying the defining requirements of $S_{i,j}$ and $S_{i'j'}$ such that
$$(iq-jq^2+k_1)(q^3-q+1)+(jq^2-i+k_2)(q^3+1)=(i'q-j'q^2+k'_1)(q^3-q+1)+(j'q^2-i'+k'_2)(q^3+1).$$
As above, we have equation \eqref{eq:ijk} as well as the similar equation
$$i'q-j'q^2+k'_1 <0 \ \makebox{and} \ 0 \le j'q^2-i'+k'_2 \le q^3-q.$$
Observation \ref{obs:representationab} therefore implies that
$$iq-jq^2+k_1=i'q-j'q^2+k'_1 \ \makebox{and} \ jq^2-i+k_2=j'q^2-i'+k'_2,$$
and in particular $(i-i')q-(j-j')q^2+(k_1-k_1')=0.$ Considering this equation modulo $q$ and modulo $q^2$, we see that $k_1=k_1'$ and $i=i'$, implying that $j=j'$ as well. Then it is also clear that $k_2=k_2'$. This implies the second item.

As for the third item: if
$$(iq-jq^2+k_1)(q^3-q+1)+(jq^2-i+k_2)(q^3+1)=(iq-jq^2+k'_1)(q^3-q+1)+(jq^2-i+k'_2)(q^3+1),$$ with
integers $k_1,k_1',k_2,k_2'$ satisfying the defining requirements of $S_{i,j}$, then the same reasoning as in above proof of the second item, shows that $k_1=k_1'$ and $k_2=k_2'$. Hence the cardinality of $S_{i,j}$ is simply the number of possibilities for $k_1$ times that for $k_2$.
\end{proof}

Picture \ref{fig1} describes the sets $S_{i,j}$ for $q=3$. In this picture a point of coordinates $(a,b)$ is used to represent the element $a(q^3-q+1)+b(q^3+1)$. Black dots represent elements of the numerical semigroup $S$, while white dots represent the elements contained in $S_{i,j}$ for some $i$ and $j$.

\begin{figure}[H]
\label{fig1}
 \centering
  \begin{tikzpicture}[x=2.5cm,y=2.5cm]
    \draw[->] (-2.8,0)--(3,0) node [right] {$a$};
    \draw[->] (0,0)--(0,2.7) node [above] {$b$};
    \draw[-] (-2.8,2.4)--(3,2.4) node [right] {$24$};
    \draw[-] (0,0)--(-2.8,2.4);

\foreach \x in {0,0.1,...,2.8}
      \foreach \y in {0,0.1,...,2.5}
        \node[draw,circle,inner sep=1pt,fill] at (\x,\y) {};

 \node[draw,circle,inner sep=1pt] at (-0.3,0.7) {};
 \node[draw,circle,inner sep=1pt] at (-0.9,0.9) {};
 \node[draw,circle,inner sep=1pt] at (-0.6,0.8) {};
 \node[draw,circle,inner sep=1pt] at (-1.2,1.6) {};
 \node[draw,circle,inner sep=1pt] at (-1.5,1.7) {};
 \node[draw,circle,inner sep=1pt] at (-1.8,1.8) {};

\draw[-] (1.3,2.5) --(1.3,2.5)  node [above, right]{$S$};
\draw[-] (-0.3,0.7)--(-0.3,2.5) node [above, right]{$S_{2,1}$};
\draw[-]  (-0.9,0.9) --(-0.9,2.5)  node [above, right]{$S_{0,1}$};
\draw[-]  (-0.6,0.8) --(-0.6,2.5)  node [above, right]{$S_{1,1}$};
 \draw[-]  (-1.2,1.6) --(-1.2,2.5)  node [above, right]{$S_{2,2}$};
\draw[-]  (-1.5,1.7) --(-1.5,2.5)  node [above, right]{$S_{1,2}$};
\draw[-] (-1.8,1.8) --(-1.8,2.5)  node [above, right]{$S_{0,2}$};

\draw[-]  (-0.3,0.7)--(0,0.7);
\draw[-] (-0.9,0.9) --(-0.6,0.9);
\draw[-] (-0.6,0.8) --(-0.3,0.8);
 \draw[-]  (-1.2,1.6) --(-0.9,1.6);
\draw[-] (-1.5,1.7) --(-1.2,1.7);
\draw[-] (-1.8,1.8) --(-1.5,1.8);
\draw[style=help lines,dashed] (-2.8,2.4) --(-2.8,0) node [below] { \textcolor{black}{$-28$}};

\foreach \x in {-0.1,-0.2,-0.3}
      \foreach \y in {0.7,0.8,...,2.5}
        \node[draw,circle,inner sep=1pt] at (\x,\y) {};

\foreach \x in {-0.4,-0.5,-0.6}
      \foreach \y in {0.8,0.9,...,2.4}
        \node[draw,circle,inner sep=1pt] at (\x,\y) {};

\foreach \x in {-0.7,-0.8,-0.9}
      \foreach \y in {0.9,1,...,2.5}
        \node[draw,circle,inner sep=1pt] at (\x,\y) {};

\foreach \x in {-1,-1.1,-1.2}
      \foreach \y in {1.6,1.7,...,2.4}
        \node[draw,circle,inner sep=1pt] at (\x,\y) {};

\foreach \x in {-1.3,-1.4,-1.5}
      \foreach \y in {1.7,1.8,...,2.5}
        \node[draw,circle,inner sep=1pt] at (\x,\y) {};

\foreach \x in {-1.6,-1.7,-1.8}
      \foreach \y in {1.8,1.9,...,2.4}
        \node[draw,circle,inner sep=1pt] at (\x,\y) {};

\draw[style=help lines,dashed] (-0.3,0.7) --(-0.3,0) node [below] { \textcolor{black}{$-3$}};
\draw[style=help lines,dashed] (-0.6,0.8) --(-0.6,0) node [below] { \textcolor{black}{$-6$}};
\draw[style=help lines,dashed] (-0.9,0.9) --(-0.9,0) node [below] { \textcolor{black}{$-9$}};
\draw[style=help lines,dashed] (-1.2,1.6) --(-1.2,0) node [below] { \textcolor{black}{$-12$}};
\draw[style=help lines,dashed] (-1.5,1.7) --(-1.5,0) node [below] { \textcolor{black}{$-15$}};
\draw[style=help lines,dashed] (-1.8,1.8) --(-1.8,0) node [below] { \textcolor{black}{$-18$}};

\draw[style=help lines,dashed] (-0.3,0.7) --(-2.8,0.7) node [left] {\textcolor{black}{$7$}};
\draw[style=help lines,dashed] (-0.6,0.8) --(-2.8,0.8);
\draw[style=help lines,dashed] (-0.9,0.9) --(-2.8,0.9) node [left] { \textcolor{black}{$9$}};
\draw[style=help lines,dashed] (-1.2,1.6) --(-2.8,1.6) node [left] { \textcolor{black}{$16$}};
\draw[style=help lines,dashed] (-1.5,1.7) --(-2.8,1.7);
\draw[style=help lines,dashed] (-1.8,1.8) --(-2.8,1.8) node [left] { \textcolor{black}{$18$}};

\end{tikzpicture}
\caption{The sets $S_{i,j}$ and $S$ for $q=3$}
  \label{figure}

\begin{tabular}{r@{: }l r@{: }l}
$\circ$ & Elements in $S_{i,j}$ & \textbullet & Elements in $S$
\end{tabular}
\end{figure}

We are now ready to prove Conjecture \ref{conjD}.
\begin{theorem}
We have $g(T)=g(\cX)$ and in particular $H(P)=T.$
\end{theorem}

\begin{proof}
Proposition \ref{prop:contained} implies that $g(T) \ge g(\cX)$. Hence the theorem follows once we show that $g(T) \le g(\cX)$. However, using the first two items of Lemma \ref{lem:sij}, we see that $$g(T) \le g(S)-\sum_{i=0}^{q-1}\sum_{j=1}^{q-1}|S_{i,j}|.$$
Using Lemma \ref{genusS} and item three of Lemma \ref{lem:sij} we obtain
\begin{align*}
g(T) &\leq  \frac{q^6-q^4}{2} -\sum_{i=0}^{q-1} \sum_{j=1}^{q-1} q(q^3-q+1-jq^2+i)\\
  &=  \frac{q^6-q^4}{2} -\sum_{i=0}^{q-1} \sum_{j=1}^{q-1} q(q^3-q+1) +\sum_{i=0}^{q-1} \sum_{j=1}^{q-1}jq^3-\sum_{i=0}^{q-1} \sum_{j=1}^{q-1}iq\\
 &=  \frac{q^6-q^4}{2}-q^2(q-1)(q^3-q+1)+\frac{q^5(q-1)}{2}-\frac{q^2(q-1)^2}{2}=\frac{q^5-2q^3+q^2}{2}=g(\cX).
\end{align*}

\end{proof}

A direct consequence of the above theorem is that $H(P)=\left(\bigcup_{i,j}S_{i,j} \right) \cup S$. It is not hard to obtain more information about $H(P)$ from the above calculations. For example, it is clear that the multiplicity of $H(P)$ (i.e., the smallest positive element in $H(P)$) is equal to $q^3-q+1$, while its conductor (i.e., the largest gap) is $2g(\cX)-1$. This means in particular that like $H(P_\infty)$, the semigroup $H(P)$ is symmetric. Since $H(P_\infty)$ has multiplicity $q^3-q^2+q$, we also see that $H(P) \neq H(P_\infty).$ %In particular, since all the elements in $\mathbb{N}$ are then contained in the union of $S$ and the triangle $\Delta$ on the left-hand side of the affine plane, this will show that the set of gaps is given by the elements that correspond to $(a,b) \in \Delta \setminus\bigg(\bigcup_{i,j}S_{i,j} \bigg)$.

\section{The Weierstrass semigroup $H(P)$ for $P \not\in \cX(\mathbb{F}_{q^6})$}

In this section we determine the Weierstrass semigroup $H(P)$ for $P \not\in \cX(\mathbb{F}_{q^6})$. In particular in this section $P=P_{(a,b,c)}$ will always denote a point on $\cX$ not in $\cX(\mathbb{F}_{q^6})$. For future reference, note that as in the previous section, this means that $c \neq 0$. As we will see, the semigroup $H(P)$ is the same for all $P \not\in \cX(\mathbb{F}_{q^6})$ and hence the `generic' semigroup for a point on $\cX$. Our approach is use Corollary \ref{holom} to construct gaps of $H(P)$ by computing the valuation at $P$ of functions $f \in L((2g(\cX)-2)P_\infty).$ It is very easy to find a basis of the Riemann--Roch space $L((2g(\cX)-2)P_\infty)$. For example the functions $x^iy^jz^k$ where $i \ge 0$, $0 \le j \le  q$, $0 \le k \le q^2+q$ and $i(q^3+1)+j(q^3-q^2+q)+kq^3\le 2g(\cX)-2$ form a basis. However, this does not settle the matter, since these basis elements all will have valuation $0$ at $P$. Therefore an effort must be made to construct functions in $L((2g(\cX)-2)P_\infty)$ having distinct valuations at $P$. In the next subsection, we construct functions with various valuations at $P$. After that we will combine these functions and obtain a set $G$ of several explicitly described gaps of $H(P)$ using Corollary \ref{holom}. The remainder of the section will then be a somewhat lengthy calculation showing that the set $G$ in fact contains $g(\cX)$, and hence all, gaps of $H(P)$.

\subsection{Construction of functions.}

We start by constructing a function $g_1$ with small, but positive, valuation at $P=P_{(a,b,c)}$. It will be convenient to define $\beta=b^{q^2}-b$. Note that $b^{q^2}-b=c^{q^2-q+1} \neq 0$, since $P \not \in \cX(\fqs)$ (and therefore a fortiori $P \not \in \cX(\fqq)$). We define
\begin{equation*}
g_1:=(\beta^{q^2-1}-1)\tilde x_P^q+\beta^{q^2+q}+\beta^{q}\left((\tilde y_P-\beta)(\tilde x_P+\beta^q(\tilde y_P-\beta))^{q-1}\right).
\end{equation*}
The functions $\tilde x_P$ and $\tilde y_P$ are as in equation \eqref{tilda}.
This definition may seen ad hoc, but it arises naturally when constructing functions of low pole order at $P_\infty$ and large vanishing order at $P$. More precisely, we have the following lemma.

\begin{lemma}\label{lem:g1}
The function $g_1$ is an element of $L((2g(\cX)-2)P_\infty)$. Moreover $v_{P_\infty}(g_1)\ge-q(q^3+1)$ and $v_P(g_1)=q^2+1$.
\end{lemma}
\begin{proof}
It is clear that $g_1$ only can have a pole at $P_\infty$. Moreover, from equations \eqref{eq:divxpX} and \eqref{eq:divypX} imply that $\tilde x_P$ (resp. $\tilde y_P$) has a pole at $P_\infty$ of order $q^3+1$ (resp. $q^3-q^2+q$). Therefore, the triangle inequality implies that $v_{P_\infty}(g_1) \ge v_{P_\infty}(\tilde x_P^q)=-q(q^3+1),$ which is what we want to show. %Note that equality does not have to hold. The coefficient of $x_P^q$, $\beta^{q^2-1}-1$, may be $0$, since we are in the generic case now.

From equation \eqref{eq:divypX}, we see that the function $\tilde y_P$ is a local parameter for the point $P=P_{(a,b,c)}$. The defining equation for $\mathcal H_q$ directly implies that $\tilde x_P^q+\tilde x_P=\beta \tilde y_P^q-\tilde y_P^{q+1}$. Hence we easily can obtain the power series development of $\tilde x_P$ in terms of $\tilde y_P$. More precisely, we obtain that
\begin{align}\label{eq:powerx}
\tilde x_P & = \beta \tilde y_P^q-\tilde y_P^{q+1}-\tilde x_P^q=\beta \tilde y_P^q-\tilde y_P^{q+1}-\beta^q \tilde y_P^{q^2}+\tilde y_P^{q^2+q}+\cdots \notag\\
  & = (\tilde y_P-\beta)(-\tilde y_P^q+(\tilde y_P-\beta)^{q-1}\tilde y_P^{q^2})+\cdots
\end{align}
Using this, we also obtain that
\begin{align}\label{eq:powerw}
(\tilde y_P-\beta)\left(\tilde x_P+\beta^q(\tilde y_P-\beta)\right)^{q-1}&=(\tilde y_P-\beta)\left((\tilde y_P-\beta)(-\tilde y_P^q+(\tilde y_P-\beta)^{q-1}\tilde y_P^{q^2}) + \beta^q(\tilde y_P-\beta) \right)^{q-1}+\cdots\notag\\
& = (\tilde y_P-\beta)^q\left( -(\tilde y_P-\beta)^q+(\tilde y_P-\beta)^{q-1}\tilde y_P^{q^2}\right)^{q-1}+\cdots\notag\\
&=(\tilde y_P-\beta)^{q^2-q+1}\left( -(\tilde y_P-\beta)+\tilde y_P^{q^2}\right)^{q-1}+\cdots\notag\\
&=(\tilde y_P-\beta)^{q^2}-(\tilde y_P-\beta)^{q^2-1}\tilde y_P^{q^2}+\cdots\notag\\
&=-\beta^{q^2}+(1-\beta^{q^2-1})\tilde y_P^{q^2}+\beta^{q^2-2}\tilde y_P^{q^2+1}+\cdots.
\end{align}
Combining equations \eqref{eq:powerx} and \eqref{eq:powerw}, we see that
\begin{align*}
g_1&=(\beta^{q^2-1}-1)\beta^q \tilde y_P^{q^2}+\beta^{q^2+q}+\beta^q(-\beta^{q^2}+(1-\beta^{q^2-1})\tilde y_P^{q^2}+\beta^{q^2-2}\tilde y_P^{q^2+1}) + \cdots\\
&=\beta^{q^2+q-2}\tilde y_P^{q^2+1}+\cdots
\end{align*}
This implies that $v_P(g_1)=q^2+1$, which is what we wanted to show.
\end{proof}

The next functions are inspired by the previous section in the sense that we again use the functions $\tilde x_P^{(k)}$ introduced in equation \eqref{eq:FrobxP}, but now for $P=P_{(a,b,c)} \not\in \cX(\fqs)$. For $s=1,\ldots,q-2$ we define
$$h_s:=\left( \frac{\tilde x_P^q}{\tilde x_P^{(1)}} \right)^{s+1} \cdot \tilde x_P^{(2)}.$$
We have the following lemma about these functions.
\begin{lemma}\label{lem:hs}
Let $s=1,\dots,q-2$. The function $h_s$ is an element of $L((2g(\cX)-2)P_\infty)$. Moreover $v_{P_\infty}(h_s)=-(q(s+1)-s)(q^3+1)$ and $v_P(h_s)=(s+1)q^2$.
\end{lemma}
\begin{proof}
Using equations \eqref{eq:divxpX} and \eqref{eq:divFrobxP}, we see that $v_{P_\infty}(h_s)=-(q(s+1)-s)(q^3+1)$ and that $h_s$ has no other poles. Further it is well known that $\mathcal H_q(\fqq)=\mathcal H_q(\mathbb{F}_{q^4}).$ Since any point in $\mathcal H_q(\fqq)$ ramifies totally in the cover $\cX \to \cH$, this means that also $\mathcal \cX(\fqq)=\mathcal \cX(\mathbb{F}_{q^4}).$ Therefore $v_P(\tilde x_P^{(2)})=0$, since $P \not \in \cX(\fqs)$. This implies that
$$v_P(h_s)=(s+1)\left(qv_P(\tilde x_P)-v_P(\tilde x_P^{(1)})\right)=(s+1)q^2,$$ as claimed.
\end{proof}
Now we able to determine several gaps of $H(P)$.
\begin{proposition} \label{gaps}
Let $P \not\in \cX(\mathbb{F}_{q^6})$ be a point on $\cX$. Then
\begin{multline}
G:=\{iq^3+j+kq+m(q^2+1)+\sum_{s=1}^{q-2} n_s ((s+1)q^2)+1 \mid i,j,k,m,n_1,\ldots,n_{q-2} \in \mathbb{Z}_{\geq 0}, \ \makebox{and}\\
i(q+1)+jq+k(q+1)+mq(q+1)+\sum_{s=1}^{q-2} n_s ((s+1)q-s)(q+1) \leq (q+1)(q^2-2)\},\notag
\end{multline}
is a set of gaps at $P$.
\end{proposition}
\begin{proof}
Let $i,j,k,m,n_1,\ldots,n_{q-2}$ be nonnegative integers and write $f= \tilde z_P^i\tilde y_P^j\tilde x_P^k  g_1^m \prod_{s=1}^{q-2} h_s^{n_s}$. Equations \eqref{eq:divxpX}, \eqref{eq:divypX}, \eqref{eq:divzpX} combined with Lemmas \ref{lem:g1} and \ref{lem:hs} imply that $f \in L((2g(\cX)-2)P_\infty)$ if
$$i(q^3+1)+j(q^3-q^2+q)+k(q^3+1)+m(q^4+q)+\sum_{s=1}^{q-2} n_s ((s+1)q-s)(q^3+1) \leq q^5-2q^3+q^2-2,$$
which is equivalent to
\begin{equation}\label{eq:inquality1}
i(q+1)+jq+k(q+1)+mq(q+1)+\sum_{s=1}^{q-2} n_s ((s+1)q-s)(q+1) \leq (q+1)(q^2-2).
\end{equation}
On the other hand we have $$v_P(f)=iq^3+j+kq+m(q^2+1)+\sum_{s=1}^{q-2} n_s ((s+1)q^2).$$
Hence the claim follows from Lemma \ref{holom}.
\end{proof}

\begin{observation}\label{obs:largestgapinG}
Inequality \eqref{eq:inquality} implies in particular that $i\leq q^2-2,j \leq q^2+q-3,k \leq q^2-2$, $m \leq q-1$ and $n_s \leq \lfloor (q+1)/(s+1)\rfloor$. This implies directly that the largest gap of $H(P)$ that is contained in $G$ is obtained by putting $i=q^2-2$ and all other remaining variables to $0$. In other words: the largest element in $G$ is $q^5-2q^3+1=2g(\cX)-q^2+1.$
\end{observation}

\begin{observation}\label{obs:jsmall}
If $j\ge q$ and the tuple $(i,j,k,m,n_1,\dots,n_{q-2})$ satisfies inequality \eqref{eq:inquality}, then the tuple $(i,j-q,k+1,m,n_1,\dots,n_s)$ will also satisfy inequality \eqref{eq:inquality}. This implies that when calculating the set $G$, we may assume that $j \le q-1$. Moreover, inequality \eqref{eq:inquality1} is equivalent to
\begin{equation*}
i+j+k+mq+\sum_{s=1}^{q-2} n_s ((s+1)q-s) \leq q^2-2+\frac{j}{q+1},
\end{equation*}
which for $j \le q-1$ is equivalent to
\begin{equation}\label{eq:inquality}
i+j+k+mq+\sum_{s=1}^{q-2} n_s ((s+1)q-s) \leq q^2-2,
\end{equation}
since all variables involved are integers.
\end{observation}

\subsection{$|G|=g(\cX)$.}

We now prove that $G$ is exactly the set of gaps $G$ at $P=P_{(a,b,c)} \not\in \cX(\fqs)$, that is $|G|=g(\cX)$. Since we already know that $G$ contains gaps of $H(P)$, it is sufficient to show that $|G| \ge g(\cX)$. This will require a detailed study of the elements of $G$. To this end we consider the following map
$$\varphi: \mathbb{Z}_{\geq 0}^{q+2} \rightarrow \mathbb{Z}_{\geq 0}, \quad {\rm with} \quad \varphi(i,j,k,m,n_1,\ldots,n_{q-2}) = iq^3+j+kq+m(q^2+1)+\sum_{s=1}^{q-2} n_s ((s+1)q^2)+1,
$$
and consider the set
$$\mathcal{G}=\{(i,j,k,m,n_1,\ldots,n_{q-2}) \in \mathbb{Z}_{\geq 0}^{q+2} \mid j \le q-1, \ \makebox{inequality \eqref{eq:inquality} holds}\}.$$
%\begin{multline}\mathcal{G}=\{(i,j,k,m,n_1,\ldots,n_{q-2}) \in \mathbb{Z}_{\geq 0}^{q+2} \mid \\ i(q+1)+jq+k(q+1)+mq(q+1)+\sum_{s=1}^{q-2} n_s ((s+1)q-s)(q+1)\leq (q+1)(q^2-2)\},\notag
%\end{multline}
Then by Observation \ref{obs:jsmall} we have $G=\varphi(\mathcal{G})$. The main difficulty is that $\varphi_{\big | \mathcal{G}}$, the restriction of the map $\varphi$ to $\mathcal{G}$, is not injective. This makes estimating the cardinality of $G$ somewhat tricky. We proceed by studying the image of $\varphi$ on the following three subsets of $\mathcal G$.
\begin{align*}
\mathcal G_1&:=\{(i,0,k,m,0,\dots,0) \in \mathcal G\},\\
\mathcal G_2&:=\{(i,j,k,m,0,\dots,0) \in \mathcal G \mid 1 \le j \le q-1, k \le q-1, j+m \le q-1\}\\
\mathcal G_3&:=\{(i,j,k,0,\dots,0,n_s,0,\dots,0) \in \mathcal G \mid k \le q-1,1 \le s \le q-2,n_s=1,i+k+(s+1)q \ge q^2-1\}.
\end{align*}
Further, we write $G_1=\varphi(\mathcal G_1)$, $G_2=\varphi(\mathcal G_2)$ and $G_3=\varphi(\mathcal G_3)$. We will show that these sets are mutually disjoint and that their cardinalities add up to $|G|$ in a series of lemmas.

\begin{lemma}\label{lem:G1}
Let $\mathcal G_1$ and $G_1=\varphi(\mathcal{G}_1)$ be as above. Then $\varphi$ restricted to $\mathcal G_1$ is injective and
$$|G_1|=\frac12 q^2(q-1)\left( \frac13 q^2+\frac56 q + \frac12\right).$$
\end{lemma}
\begin{proof}
If $(i,0,k,m,0,\dots,0) \in \mathcal G_1$, then $\varphi(i,0,k,m,0,\dots,0)=iq^3+kq+m(q^2+1)+1$ and by inequality \eqref{eq:inquality} $i+k+mq \le q^2-2.$ This implies in particular that
$$0 \le m \le q-1 \ \makebox{and} \ 0 \le kq+m(q^2+1) \le (k+mq)q+q-1 \le (q^2-2)q+q-1 < q^3.$$
Now suppose $(i_1,0,k_1,m_1,0,\dots,0),(i_2,0,k_2,m_2,0,\dots,0) \in \mathcal G_1$ and $$i_1q^3+k_1q+m_1(q^2+1)=i_2q^3+k_2q+m_2(q^2+1).$$ Calculating modulo $q$ and using that $0 \le m_1 \le q-1$ and $0 \le m_2 \le q-1$ (see Observation \ref{obs:jsmall}), we see that $m_1=m_2$. Further, since $0 \le k_1q+m_1(q^2+1)<q^3$ and $0 \le k_2q+m_2(q^2+1)<q^3$, we see that $k_1q+m_1(q^2+1)=k_2q+m_2(q^2+1)$ and $i_1q^3=i_2q^3$. Combining these equalities, we see that $(i_1,0,k_1,m_1,0,\dots,0)=(i_2,0,k_2,m_2,0,\dots,0)$, which is what we wanted to show.

Now we compute $|G_1|.$ First of all, from the above we see that $|G_1|=|\mathcal G_1|$. Further we have
\begin{align*}
|\mathcal G_1|&= \sum_{m=0}^{q-1}\sum_{i=0}^{q^2-2-mq}\sum_{k=0}^{q^2-2-mq-i} 1=\sum_{m=0}^{q-1}\sum_{i=0}^{q^2-2-mq} (q^2-1-mq-i)\\
& = \sum_{m=0}^{q-1}\frac{(q^2-1-mq)(q^2-mq)}{2} = \frac{(q^2-1)q^3}{2}+\sum_{m=0}^{q-1}\frac{-2q^3-q^2+q}{2} m+\binom{m+1}{2}q^2\\
& = \frac{(q^2-1)q^3}{2}+\frac{-2q^3-q^2+q}{2}\binom{q}{2}+\binom{q+1}{3}q^2.
\end{align*}
In the last equality we used \textit{summation on the upper index} to evaluate the summation $\sum_m\binom{m+1}{2}$; see \cite[Eqn. (5.10)]{GKP}.
The desired equality for $|G_1|$ now follows.
\end{proof}

\begin{lemma}\label{lem:G2}
Let $\mathcal G_2$ and $G_2=\varphi(\mathcal{G}_2)$ be as above. Then $\varphi$ restricted to $\mathcal G_2$ is injective and
$$|G_2|=\frac12 q^2(q-1)\left( \frac23 q^2-\frac16 q - \frac56\right).$$
\end{lemma}
\begin{proof}
If $(i,j,k,m,0,\dots,0) \in \mathcal G_2$, then $\varphi(i,j,k,m,0,\dots,0)=iq^3+j+kq+m(q^2+1)+1$ and by definition we have $1 \le j \le q-1$, $1 \le j+m \le q-1$ and $0 \le k \le q-1$. Moreover, inequality \eqref{eq:inquality} gives that $i+j+k+mq \le q^2-2.$
Similarly as in the previous lemma, we obtain that
$$0 \le m \le q-1 \ \makebox{and} \ 0 \le j+kq+m(q^2+1) \le (k+mq)q+q-1 \le (q^2-2)q+q-1 < q^3.$$
Now suppose $(i_1,j_1,k_1,m_1,0,\dots,0),(i_2,j_2,k_2,m_2,0,\dots,0) \in \mathcal G_2$ and $$i_1q^3+j_1+k_1q+m_1(q^2+1)=i_2q^3+j_2+k_2q+m_2(q^2+1).$$ Reasoning exactly as in the previous lemma, we obtain that $j_1+m_1=j_2+m_2$, $j_1+k_1q+m_1(q^2+1)=j_2+k_2q+m_2(q^2+1)$ and $i_1=i_2$. Combining the first two equations, we deduce that $k_1q+m_1q^2=k_2q+m_2q^2$. Since $0\le k_1 \le q-1$ and $0 \le k_2 \le q-1$, we see $k_1=k_2$, which now implies that $(i_1,j_1,k_1,m_1,0,\dots,0)=(i_2,j_2,k_2,m_2,0,\dots,0).$

Now we compute $|G_2|$. First note that $k\le q-1$, but for a given $j$ and $m$, we also have $k \le q^2-2-j-mq$. However, since $j \ge 1$ and $0 \le j+m \le q-1$, we see that $m \le q-2$. Hence $q^2-2-j-mq \ge q^2-2-1-(q-2)q \ge q-1,$ implying that the condition $k \le q^2-2-j-mq$ is trivially satisfied. Hence
\begin{align*}
|\mathcal G_2|&= \sum_{j=1}^{q-1}\sum_{m=0}^{q-1-j}\sum_{k=0}^{q-1} \sum_{i=0}^{q^2-2-j-k-mq}1=\sum_{j=1}^{q-1}\sum_{m=0}^{q-1-j}\sum_{k=0}^{q-1} (q^2-1-j-k-mq)\\
& = \sum_{j=1}^{q-1}\sum_{m=0}^{q-1-j} (q^2-1-j-mq)q-\binom{q}{2}=\sum_{j=1}^{q-1}\left((q^2-1-j)q-\binom{q}{2}\right)(q-j)-q^2\binom{q-j}{2}\\
& = \sum_{j=1}^{q-1}\left((q^2-q)q-\binom{q}{2}\right)(q-j)-(q^2-2q)\binom{q-j}{2}
=\left((q^2-q)q-\binom{q}{2}\right)\binom{q}{2}-(q^2-2q)\binom{q}{3}.
\end{align*}
The desired equality now follows.
\end{proof}

\begin{lemma}\label{lem:G3}
Let $\mathcal G_3$ and $G_3=\varphi(\mathcal{G}_3)$ be as above. Then $\varphi$ restricted to $\mathcal G_3$ is injective and
$$|G_3|=\frac12 q^2(q-1)\left( \frac13 q - \frac23\right).$$
\end{lemma}
\begin{proof}
If $(i,j,k,0,0,\dots,0,n_s,0,\dots,0) \in \mathcal G_3$, then $\varphi(i,j,k,0,0,\dots,0,n_s,0,\dots,0)=iq^3+j+kq+(s+1)q^2+1$ and by definition we have $n_s=1$, $1 \le s \le q-2$, $0 \le j \le q-1$, $0 \le k \le q-1$ and $i+k+(s+1)q \ge q^2-1$ (that is $i+k+sq \ge q^2-q-1$). Moreover, inequality \eqref{eq:inquality} gives that $i+j+k+s(q-1) \le q^2-q-2.$ Note that the inequalities $i+k+sq \ge q^2-q-1$ and $i+j+k+s(q-1) \le q^2-q-2$ only can be satisfied simultaneously, if $j \le s-1$, so we may assume this as well in the remainder of the proof.

Now suppose $(i_1,j_1,k_1,0,0,\dots,0,n_s,0,\dots,0),(i_1,j_1,k,0,0,\dots,0,1,0,\dots,0) \in \mathcal G_3$ and $$i_1q^3+j_1+k_1q+(s_1+1)q^2=i_2q^3+j_2+k_2q+(s_2+1)q^2.$$ Since the $q$-ary expansion of a number is unique, we immediately obtain that $j_1=j_2$, $k_1=k_2$ and $s_1=s_2$, since all variables involved at between $0$ and $q-1$. Hence $i_1=i_2$ as well and the first part of the lemma follows.

Now we compute $|G_3|$. Recall that we may assume $j \le s-1$. Hence
\begin{align*}
|\mathcal G_3|&= \sum_{s=1}^{q-2}\sum_{j=0}^{s-1}\sum_{k=0}^{q-1}\sum_{i=q^2-q-1-k-sq}^{q^2-q-2-j-k-s(q-1)}1=
\sum_{s=1}^{q-2}\sum_{j=0}^{s-1}\sum_{k=0}^{q-1}(s-j)\\
& = q\sum_{s=1}^{q-2}\sum_{j=0}^{s-1}(s-j)=q\sum_{s=1}^{q-2}\binom{s+1}{2}=q\binom{q}{3}.
\end{align*}
The desired equality now follows.
\end{proof}

Finally to obtain an estimate for $|G|$, we need to study the intersections of the sets $G_1$, $G_2$ and $G_3$. It turns out that they are disjoint, as we will now show.

\begin{lemma}\label{G1G2G3disjoint}
The sets $G_1$, $G_2$ and $G_3$ defined above are mutually disjoint.
\end{lemma}
\begin{proof}
\noindent
{\bf Part 1. $G_1 \cap G_2 = \emptyset$.}
Let $(i_1,0,k_1,m_1,0,\dots,0) \in \mathcal G_1$, $(i_2,j_2,k_2,m_2,0,\dots,0) \in \mathcal G_2$ and suppose that $$i_1q^3+k_1q+m_1(q^2+1)=i_2q^3+j_2+k_2q+m_2(q^2+1).$$ Since $0 \le m_1 \le q-1$ and $1 \le j_2+m_2 \le q-1$, we see that $m_1=j_2+m_2$ and hence that $i_1q^2+k_1+m_1q=i_2q^2+k_2+m_2q.$ Note that $m_1-m_2=j_2 \ge 0$, where the inequality follows from the definition of $\mathcal G_2$. Inequality \eqref{eq:inquality} implies that $k_1+m_1q < q^2$ as well as $k_2+m_2q < q^2$. Hence we obtain $i_1=i_2$ and $k_1+m_1q=k_2+m_2q$, whence $(m_1-m_2)q=k_2-k_1$. This implies that $k_1 \equiv k_2 \pmod{q}$, but since $k_1 \ge 0$ and $0 \le k_2 \le q-1$ we can deduce $k_1-k_2 \ge 0$. On the other hand we already have seen that $m_1-m_2=j_2 \ge 1$, but then we arrive at a contradiction, since $0<(m_1-m_2)q=k_2-k_1 \le 0$.

\noindent
{\bf Part 2. $G_1 \cap G_3 = \emptyset$.}
Let $(i_1,0,k_1,m_1,0,\dots,0) \in \mathcal G_1$, $(i_3,j_3,k_3,0,0,\dots,0,1,0,\dots,0) \in \mathcal G_3$ and suppose that $$i_1q^3+k_1q+m_1(q^2+1)=i_3q^3+j_3+k_3q+(s+1)q^2.$$ Similarly as in part 1 above, we obtain that $m_1=j_3$, whence $i_1q^2+k_1+m_1q=i_3q^2+k_3+(s+1)q$, as well as the inequality $k_1+m_1q<q^2$. However, since $k_3 \le q-1$ and $s+1 \le q-1$, we also have $k_3+(s+1)q<q^3$. Therefore we obtain that $i_1=i_3$ as well as $k_1+m_1q=k_3+(s+1)q$. This implies that $$i_3+k_3+(s+1)q=i_1+k_1+m_1q \le q^2-2,$$
where we have used inequality \eqref{eq:inquality} to obtain the inequality. On the other hand $i_3+k_3+(s+1)q \ge q^2-1$ by the definition of $\mathcal G_3$ and we arrive at a contradiction.

\noindent
{\bf Part 3. $G_2 \cap G_2 = \emptyset$.}
Let $(i_2,j_2,k_2,m_2,0,\dots,0) \in \mathcal G_2$, $(i_3,j_3,k_3,0,0,\dots,0,1,0,\dots,0) \in \mathcal G_3$ and suppose that $$i_2q^3+j_2+k_2q+m_2(q^2+1)=i_3q^3+j_3+k_3q+(s+1)q^2.$$
Reasoning very similarly as in Part 1 and Part 2, we obtain $j_2+m_2=j_3$, $i_2=i_3$ and $$i_3+k_3+(s+1)q=i_2+k_2+m_2q \le q^2-2.$$ Again we arrive at a constriction.
\end{proof}

We are now ready to prove the main theorem of this section.
\begin{theorem}
Let $P$ be a point of $\cX$ with $P\not\in \cX(\mathbb{F}_{q^6})$. Then the set of gaps of $H(P)$ is given by,

$$G=\{iq^3+kq+m(q^2+1)+\sum_{s=1}^{q-2} n_s ((s+1)q^2)+j+1 \mid i,j,k,m,n_1,\ldots,n_{q-2} \in \mathbb{Z}_{\geq 0}, j \le q-1,\ \makebox{and}$$
$$ i+j+k+mq+\sum_{s=1}^{q-2} n_s ((s+1)q-s) \leq q^2-2\}.$$
Moreover, the set of Weierstrass points $W$ on $\cX$ coincides with $\cX(\mathbb{F}_{q^6})$.
\end{theorem}
\begin{proof}
Combing Lemmas \ref{lem:G1}, \ref{lem:G2}, \ref{lem:G3}, and \ref{G1G2G3disjoint} we see that $$|G| \ge |G_1|+|G_2|+|G_3|=\frac12 q^2(q-1)(q^2+q-1)=g(\cX).$$ Since we know that $H(P)$ has exactly $g(\cX)$ gaps, Proposition \ref{gaps} then implies that $H(P)=\mathbb{N} \setminus G$. From Observation \ref{obs:largestgapinG}, we deduce that the largest gap in $H(P)$ is $2g(\cX)-q^2+1$, while we already know that for any $P \in \cX(\fqs)$, the largest gap is $2g(\cX)-1$. This implies the last statement in the theorem.
\end{proof}
The proof also shows that the gaps of $H(P)$ are precisely $G_1 \cup G_2 \cup G_3$, which is convenient when checking if a particular number is a gap or not. For example, this allows us to compute the multiplicity (smallest positive element) of $H(P)$ fairly easily.
\begin{corollary}
Let $P$ be a point of $\cX$ with $P\not\in \cX(\mathbb{F}_{q^6})$. The multiplicity of $H(P)$ is equal to $q^3-1$.
\end{corollary}
\begin{proof}
From St\"ohr-Voloch Theory we know that $q^3-1$ and $q^3$ are non-gaps at $P$, since $P$ is not a Weierstrass point; see \cite[Proposition 10.9]{HKT}. It is also not difficult to verify this directly.
On the other hand, let $1 \le a \le q^3-2$ be an integer and write $a-1=c_0+c_1q+c_2q^2$ with $0\le c_t \le q-1$ for $t=1,2,3$. Then we distinguish three cases.

\noindent
{\bf Case 1. $c_2 \ge c_0$ and $(c_1,c_2) \neq (q-1,q-1)$.}
In this case a direct verification shows that $a=\varphi(0,0,c_1+(c_2-c_0)q,c_0,0\dots,0)$ and that $(0,0,c_1+(c_2-c_0)q,c_0,0\dots,0) \in \mathcal G_1$.

\noindent
{\bf Case 2. $c_2<c_0$.}
We have $a=\varphi(0,c_0-c_2,c_1,c_2,0\dots,0)$ and $(0,c_0-c_2,c_1,c_2,0\dots,0) \in \mathcal G_2$ in this case.

\noindent
{\bf Case 3. $(c_1,c_2) = (q-1,q-1)$}
Note that in this case $c_0 \le q-3,$ since $a-1=c_0+(q-1)q+(q-1)q^2 \le q^3-3$. One then checks that $a=\varphi(0,c_0,q-1,0,0,\dots,0,1)$ and that $(0,c_0,q-1,0,0,\dots,0,1) \in \mathcal G_3$.
\end{proof}

At this point seems to be reasonable to ask for the generators of the Weierstrass semigroup $H(P)$ for $P \not\in \cX(\mathbb{F}_{q^6})$. Their explicit determination seems to be a challenging task as the following examples show. In particular the number of generators of $H(P)$ seems to grow quickly with respect to $q$.

\begin{example}
\begin{itemize}
Let $P \in \cX$ such that $P \not\in \cX(\mathbb{F}_{q^6})$.
\item If $q=2$ then $g=10$ and $$G=\{1,2,3,4,5,6,9,10,11,17\}.$$ Clearly $7$ and $8$ must be generators of $H(P)$ and since $12 \not\in \langle 7,8\rangle$ and $13 \not\in \langle 7,8,12 \rangle$ we obtain that also $12$ and $13$ are generators. Note that $\langle 7,8,12,13\rangle \cap \{0, \ldots,20\}=\{7,8,12,13,14,15,16\}$ and hence also $18$ is a generator. In fact $$H(P)=\langle 7,8,12,13,18\rangle.$$
Moreover, if $P \in \cX$ then
$$H(P)=\begin{cases} \{{ 0, 6, 8, 9, 12, 14, 15, 16, 17, 18,20,\ldots }\}, \ \makebox{if} \ P \in \cX(\mathbb{F}_{4}), \\ \{ 0, 7, 8, 9, 13, 14, 15, 16, 17, 18,20,\ldots\}, \ \makebox{if} \ P \in \cX(\mathbb{F}_{64}) \setminus \cX(\mathbb{F}_4), \\ \{0,7,8,12,13,14,15,16,18,19,20 \ldots\}, \ \makebox{otherwise}. \end{cases}$$
\item If $q=3$ then $g=99$ and
$$G=\{ 1, 2, 3, 4, 5, 6, 7, 8, 9, 10, 11, 12, 13, 14, 15, 16, 17, 18, 19, 20, 21, 22,
23, 24, 25, 28, 29, 30, 31, 32, 33, 34, $$
$$35, 36, 37, 38, 39, 40, 41, 42, 43, 44,
45, 46, 47, 48, 49, 55, 56, 57, 58, 59, 60, 61, 62, 63, 64, 65, 66, 67, 68, 69,$$
$$70, 71, 73, 82, 83, 84, 85, 86, 87, 88, 89, 90, 91, 92, 93, 94, 95, 109, 110,
111, 112, 113, 114, 115, 116, 118, 119,$$
$$ 136, 137, 138, 139, 140, 142, 163, 164,
166, 190 \}.$$
Arguing as for the previous case, one can prove that
$$H(P)=\langle 26,27,50,51,72,74,75,96,97,117,120,121,141,145,165\rangle.$$
\end{itemize}
It is unclear what the number of generators for general $q$ is. For $q=4$ the semigroup turns out to have $28$ generators.
\end{example}

Collecting the results in the paper, we have proven Theorem \ref{mainth} and Corollary \ref{mainth2} from the introduction. We finish by summing up some further facts on the various semigroups on $\cX$ in a table, leaving a question mark for the minimal number of generators in the case $P \not\in \cX(\fqs)$. Determining this number could be interesting future work.
\\

\begin{tabular}{l|l|l|l}
$P$               & multiplicity & conductor   & number of generators\\
\hline
$P \in \cX(\fqq)$ & $q^3-q^2+q$  & $2g(\cX)-1$ & $3$ \\
$P \in \cX(\fqs)\setminus \cX(\fqq)$ & $q^3-q+1$  & $2g(\cX)-1$ & $q+2$ \\
$P \not\in \cX(\fqs)$ & $q^3-1$  & $2g(\cX)-q^2+1$ & ?
\end{tabular}

\section*{Acknowledgments}

The first author gratefully acknowledges the support from The Danish Council for Independent Research (Grant No.~DFF--4002-00367). The second author would like to thank the Italian Ministry MIUR, Strutture Geometriche, Combinatoria e loro Applicazioni, Prin 2012 prot.~2012XZE22K and GNSAGA of the Italian INDAM.

\vspace{1ex}
\noindent
Peter Beelen

\vspace{.5ex}
\noindent
Technical University of Denmark,\\
Department of Applied Mathematics and Computer Science,\\
Matematiktorvet 303B,\\
2800 Kgs. Lyngby,\\
Denmark,\\
pabe@dtu.dk\\

\vspace{1ex}
\noindent
Maria Montanucci

\vspace{.5ex}
\noindent
Universita' degli Studi della Basilicata,\\
Dipartimento di Matematica, Informatica ed Economia,\\
Campus di Macchia Romana,\\
Viale dell' Ateneo Lucano 10,\\
85100 Potenza,\\
Italy,\\
maria.montanucci@unibas.it
    \end{document}